\documentclass[10pt]{amsart}
\usepackage{amsmath,amssymb,amsthm,amsfonts}
\usepackage{graphicx}
\usepackage{tikz}

\theoremstyle{plain}
\newtheorem{lemma}{Lemma}

\newtheorem*{theorem}{Theorem}
\newtheorem*{proposition}{Proposition}

\begin{document}

\title[]{Dirichlet eigenfunctions with \\nonzero mean value}

\author[]{Stefan Steinerberger}
\address{Department of Mathematics, University of Washington, Seattle}
\email{steinerb@uw.edu}

\author[]{Raghavendra Venkatraman}
\address{Department of Mathematics, The University of Utah, Salt Lake City}
\email{raghav@math.utah.edu}

\thanks{ SS has been supported by the NSF (DMS-2123224). RV acknowledges support from the Simons foundation (733694) and the NSF (DMS-2407592); in addition, RV warmly thanks Fanghua Lin and Bob Kohn for helpful discussions.}
\date{}

\begin{abstract} We consider Laplacian eigenfunctions on a domain $\Omega \subset \mathbb{R}^d$. Under Neumann boundary conditions, the first eigenfunction is constant and the others have mean value 0. The situation is different for Dirichlet boundary conditions: on `generic' domains, one would expect that every eigenfunction has nonzero mean value. The other extreme is the ball in $\mathbb{R}^d$, where among the first $n$ eigenfunctions only $\sim n^{1/d}$ have a mean value different from zero. We prove that this rate is sharp in \textit{any} smooth domain, up to a logarithmic factor: in any smooth domain~$\Omega$, among the first $n$ Dirichlet eigenfunctions at least $ (\log{n})^{-1/2} \cdot n^{1/d}  $ have a nonzero mean.
\end{abstract}

\maketitle

\section{Introduction}
\subsection{Introduction.} Let $\Omega \subset \mathbb{R}^d$ be a bounded domain with smooth boundary. We consider the eigenfunctions of the Laplacian
$$ -\Delta \phi_k = \lambda_k \phi_k$$
subject to Dirichlet boundary conditions $\phi_k\big|_{\partial \Omega} = 0$. We always assume that they are normalized $\|\phi_k\|_{L^2} = 1$ in $L^2(\Omega)$. It is easy to see that if we were to impose Neumann boundary conditions instead, then the first eigenfunction is constant and so, by orthogonality, all other Neumann eigenfunctions have mean value 0. The situation for Dirichlet boundary conditions is very different: the first eigenfunction $\phi_1$  does not change sign, and so it  does not have vanishing mean value. However, since $\phi_1$ is not constant, orthogonality with $\phi_1$ says little about the mean value of the other eigenfunctions. Unless $\Omega$ has symmetries, one would typically expect that every single eigenfunction $\phi_k$ has a nonzero average. We do not know whether this is the case and believe it could be an interesting problem (see also \cite{hempel}). We will be working on the opposite problem and argue that there cannot be too many eigenfunctions with mean value 0. We start with some examples.

\subsection{Two examples.} The cube $[0,1]^d$ shows that a positive proportion of Dirichlet eigenfunctions can have vanishing mean value. Without normalizing, the eigenfunctions are
$$ \phi(x) = \prod_{j=1}^{d} \sin\left( a_i \pi x_i \right) \qquad \mbox{where}~a_i \in \mathbb{N}.$$
When integrating the eigenfunction $\phi(x)$, the integral decouples into the product of $d$ one-dimensional integrals. In order for the product to be nonzero, all individual integrals have to be nonzero, which forces all $a_i$ to be even integers. The eigenfunctions with eigenvalue $\leq \lambda$ are exactly those lattice points for which $a_1^2 + a_2^2 + \dots + a_d^2 \leq \lambda^2 \pi^{-2}$ which is a sphere. The number of lattice points in $(2\mathbb{Z})^d$ has density $2^{-d}$ with respect to $\mathbb{Z}^d$. This shows that only a small fraction, $1/2^d$, of all the eigenfunctions have a mean value different from 0. Most Dirichlet eigenfunctions, asymptotically of density $1-1/2^{d}$, have vanishing mean. The ball $\mathbb{B} \subset \mathbb{R}^d$ is another interesting example. Writing 
$$ \Delta f = \frac{\partial^2 f}{\partial r^2} + \frac{d-1}{r} \frac{\partial f}{\partial r} + \frac{1}{r^2} \Delta_{\mathbb{S}^{d-1}} f$$
a separation of variables shows that Dirichlet eigenfunctions of the Laplacian on the ball $\mathbb{B}$ are given as the product of radial solutions of the Bessel equation with eigenfunction of the Laplace Beltrami operator~$\Delta_{\mathbb{S}^{d-1}}$ on $\mathbb{S}^{d-1}$ in the angular components. Since the sphere $\mathbb{S}^{d-1}$ has no boundary, the first eigenfunction of~$\Delta_{\mathbb{S}^{d-1}}$ is simply constant, and so all the remaining eigenfunctions~ of~$\Delta_{\mathbb{S}^{d-1}}$ on the sphere to have mean value 0. For a product of a Bessel function with a spherical harmonic to have nonzero mean, we thus require the spherical harmonic to be constant which leaves us with a pure radial Bessel function solving the radial part of the equation $-y''(r) - (d-1) y'(r)/r = \lambda y(r) $. Then, counting the number of zeroes of the solutions to this equation (which equals the number of such purely radial eigenfunctions) (see \cite{feng}), we conclude that among the first $n$ eigenfunctions at most $\sim n^{1/d}$ have nonzero mean value. 

\smallskip 
There is a simple reason why in \emph{any} domain~$\Omega$, there are always infinitely many Laplacian eigenfunctions whose mean value is different from zero.
\begin{proposition}
 Infinitely many Dirichlet eigenfunctions have nonzero mean value.
\end{proposition}
\begin{proof} Suppose there are only finitely many and consider the constant function $f(x) = 1$. Since the Dirichlet eigenfunctions form an orthonormal basis, we have
$$f(x) = \sum_{k=1}^{\infty} \left\langle f, \phi_k \right\rangle \phi_k(x) = \sum_{k=1}^{\infty} \left( \int_{\Omega} \phi_k(x) dx \right) \phi_k(x).$$
The sum has to range over infinitely many terms since otherwise we could write the constant function $f(x) =1$ as a finite linear combination of smooth functions that vanish on the boundary. Thus infinitely many of the eigenfunctions must have nonzero mean value.
\end{proof}

\subsection{The result} We are interested in trying to understand whether the number of eigenfunctions with nonvanishing mean could be further quantified. We were originally led to this problem through the work of the second named author on a problem in photonics (see \S\ref{ss.photonics} for a brief description). 
 
\begin{theorem} \label{t.main} Let $\Omega \subset \mathbb{R}^d$ be a bounded domain with smooth boundary and let $(\phi_k)_{k=1}^{\infty}$ denote a sequence of Laplacian eigenfunctions with Dirichlet boundary conditions. There exists a constant $c_{\Omega} > 0$ such that, for all $n$ sufficiently large,
$$ \# \left\{ 1 \leq k \leq n:  \int_{\Omega} \phi_k(x) \neq 0\right\} \geq c_{\Omega} \frac{n^{1/d}}{\sqrt{\log n}}.$$
\end{theorem}
\textbf{Remarks.} Several remarks are in order.
\begin{enumerate}
    \item We did not optimize for regularity of the boundary; presumably the result is true in much rougher domains.
    \item We note that, especially in the presence of symmetries, eigenvalues may have higher multiplicity and the choice of eigenfunctions $\phi_k$ would then not be unique. Our Theorem is independent of the particular choice of basis and true for all possible choices of an eigenbasis of $L^2(\Omega)$.
    \item  The natural expectation would be that the ball is extremal and that the correct rate is given by $\sim n^{1/d}$. One might also be tempted to conjecture that, unless~$\Omega$ is very symmetric (like a ball, or an annulus) the density of Dirichlet eigenfunctions with nonzero mean is \emph{positive} (see \S\ref{ss.questions} below for a precise formulation). The logarithmic factor is presumably an artifact of the proof; a different approach may conceivably be able to remove this factor altogether.  
\end{enumerate}

\subsection{Related questions.}\label{ss.questions} The result suggests a number of interesting related questions. Simply restricting to the case of planar domains $\Omega \subset \mathbb{R}^2$, we see that it's possible that most Dirichlet eigenfunctions have vanishing mean value (which happens when $\Omega$ is a disk, for example) and it is also possible that $3/4$ of the Dirichlet eigenfunctions have vanishing mean value, this happens when $\Omega = [0,1]^2$. 

\begin{quote}
    \textbf{Question.} What are the possible values of
 $$   \lim_{n \rightarrow \infty} \frac{1}{n}  \# \left\{ 1 \leq k \leq n:  \int_{\Omega} \phi_k(x) \neq 0 \right\} \qquad ?$$
 When $d=2$, the examples above show that the limit can be in $\left\{0, 1/4\right\}$. We expect it to be 1 for generic domains. Can it attain many more values? How do these values depend on the domain?
\end{quote}

It would be nice to have more examples: the Dirichlet eigenfunctions are explicit on an equilateral triangle \cite{pinksy} but the arising computations appear nontrivial. We note that symmetry considerations allow to deduce a little bit more information: for example, if $\Omega \subset \mathbb{R}^{d-1}$ is arbitrary, then the domain $\Omega \times [0,1] \subset \mathbb{R}^d$ has the property that at least half the Dirichlet eigenfunctions have mean value 0. We also note that the question may be rendered trivial by the unit disk: on the unit disk $B(0,1) \subset \mathbb{R}^2$ \textit{most} eigenfunctions have mean value 0. If we take an eigenvalue $\lambda_k$, then the corresponding eigenspace has one radial eigenfunction (with nonvanishing mean) and a number of other eigenfunctions that all have mean value 0.  We could perform a change of basis and replaces that basis by doing some orthogonal rotation in a subspace of that eigenspace: with that procedure we should be able to achieve any arbitrary limit for the question above. This is presumably only possible on the unit ball. The other way to interpret the question is to take among all eigenspaces the one that minimizes the above ratio. Another very natural question is whether the ball is characterized by the number of eigenfunctions with vanishing mean -- this question is so nice that we emphasize it.

\begin{quote}
    \textbf{Problem.} Let $\Omega \subset \mathbb{R}^d$ be a bounded, connected domain with the property that there exist Dirichlet eigenfunctions $(\phi_k)_{k=1}^{\infty}$ with the property that among the first $n$ eigenfunctions, there are $\leq c \cdot n^{1/d}$ eigenfunctions with vanishing mean. Does $\Omega$ need to be a ball or an annulus? More generally, what properties does $\Omega$ need to have?
\end{quote}

 Another interesting question, inspired by the fact that the density of non-zero mean value eigenfunctions on $[0,1]^d$ is $1/2^d$ is whether each `symmetry' (suitably interpreted) reduces the proportion of eigenfunctions with nonzero mean by a half. 
 
\subsection{Motivation from photonics} \label{ss.photonics} The set of Dirichlet eigenvalues whose associated eigenfunctions have nonzero mean appear in a number of problems on divergence form equations with piecewise constant, high contrast coefficients (see~\cite{BF, KV1} and references therein).  Our motivation comes from work by the second named author (jointly with Robert V. Kohn) on a class of photonic devices (novel types of waveguides, resonators, antennas) referred to as ``epsilon-near-zero'' devices. Referring the interested reader to~\cite{nader} for a survey of the physics, the mathematical problems one is led to consider, in the simplest settings, concern divergence form operators with piecewise-constant complex coefficients~$\mathrm{div}\,(\varepsilon^{-1}\nabla u)$ for the in-plane magnetic field~$u,$ with the following crucial feature: the dielectric permittivity~$\varepsilon \approx 0$ in part of the domain. Motivated by~\cite{geom}, in~\cite{KV2}, we study resonators: roughly speaking, we seek nontrivial \emph{Neumann} eigenfunctions of the operator~$\mathcal{L}_\delta := \mathrm{div}(\varepsilon^{-1}\nabla u)$ in a regular bounded domain~$A \subset \mathbb{R}^2$. Here,~$\varepsilon(x) = \delta \in \mathbb{C}, |\delta|\ll 1$ in~$A \setminus \overline{\Omega},$ where~$\Omega$ is a regular subdomain of~$A,$ and~$\varepsilon(x) = 1$ for~$x\in\Omega.$ In~\cite{KV2} we show that between every successive pair of \emph{Dirichlet eigenvalues of~$-\Delta$ in the subdomain~$\Omega$ whose associated eigenfunctions have nonzero mean,} the divergence form operator~$\mathcal{L}_\delta$ in~$A$ admits a simple Neumann eigenvalue, which depends complex analytically on~$\delta$ in a neighborhood of the origin. The question of how these Neumann eigenvalues of~$\mathcal{L}_\delta$ in the domain~$A$ are distributed leads to the problem considered in this paper. \\ 

 \textbf{Acknowledgment.} We are indebted to a very careful referee who informed us that \cite[Lemma 3.2]{hassell} in Hassell-Tao is much stronger than we originally used. This led to an improved final result. 

\section{Proof of the Theorem}
\subsection{Main idea.}
The main idea behind the proof is as follows: suppose first that ``virtually all'' Dirichlet eigenfunctions of the domain~$\Omega$ had mean value 0. Then the solution of the heat equation
$$ e^{t\Delta} f = \sum_{k=1}^{\infty} e^{-\lambda_k t} \left\langle f, \phi_k\right\rangle \phi_k\,,$$
would `nearly' preserve the mean value of its solution in a suitable sense. We define $A \subset \mathbb{N}$ to be the set of eigenfunctions with nonvanishing mean value
$$ A = \left\{n \in \mathbb{N}: \int_{\Omega} \phi_n(x) dx \neq 0 \right\}.$$
Our goal is to show that $A$ cannot be too small. A simple computation shows that
\begin{equation} \label{e.simple}
    \begin{aligned}
        \int_{\Omega} (e^{t\Delta} f)(x) dx &= \int_{\Omega} \sum_{k=1}^{\infty} e^{-\lambda_k t} \left\langle f, \phi_k\right\rangle \phi_k(x) dx \\
 &= \sum_{k \in A} e^{-\lambda_k t} \left\langle f, \phi_k\right\rangle \left( \int_{\Omega} \phi_k(x) dx \right)\,.
    \end{aligned}
\end{equation}

If $A$ was an extremely sparse set, then the evolution of mean values under the heat equation would be incredibly rigid. The idea is then to study one particular solution of the heat equation whose mean value undergoes rapid change in time. It makes sense that a good choice of an initial data to capture this effect is a function that is localized near the boundary~$\partial \Omega.$ We can then quantify how the Dirichlet boundary conditions induces the expected rapid change (the evolution of the heat equation far from the boundary does not `feel' the boundary and very nearly preserves mean values) and this change is due to eigenfunctions with nonzero mean.

\subsection{A specific solution} We will work with the characteristic function of a $\varepsilon-$neighborhood of $\partial \Omega:$
\begin{equation} \label{e.specific}
    f(x) = \begin{cases} 1 \qquad &\mbox{if}~d(x, \partial \Omega^c) \leq \varepsilon \\ 0 \qquad &\mbox{otherwise}\,. \end{cases}
\end{equation}
Since the boundary of our domain is smooth, we have that, as $\varepsilon \rightarrow 0$, 
$$ \int_{\Omega} f(x) dx = (1+o(1)) \cdot |\partial \Omega| \cdot \varepsilon\,,$$
and more precise asymptotic expansions on smooth domains (involving curvature for the next order term) are  available but are not of interest here. The key observation, formalized in the next lemma, is that there is substantial loss of $L^1-$mass of the solution of the heat equation at time scale $t \sim \varepsilon^2$.

\begin{lemma} \label{l.mainstep}
Let~$\Omega$ be a~$C^1$ bounded domain in~$\mathbb{R}^d.$ There exist three constants $0 < c_1 < c_2$ and $c_3 > 0$ (depending on $\Omega$) such that, for all $\varepsilon > 0$ sufficiently small (depending on $\Omega$), we have
$$ \int_{\Omega} (e^{c_1 \varepsilon^2 \Delta} f)(x) -   (e^{c_2 \varepsilon^2 \Delta} f)(x)dx \geq  c_3 \varepsilon\,.$$
\end{lemma}

We first discuss the proof. It uses the stochastic interpretation and the use of the reflection principle for Brownian motion and is not new, see for example \cite{berg, lou, lu}. Afterwards, we give a probability-free argument why one would expect this scaling.

\begin{proof}[Proof of Lemma 1]
We will rewrite the integral in terms of the Dirichlet heat kernel and then use the fact that the function $f$ is very simple to simplify the expression 
$$  (e^{t\Delta}f)(x) = \int_{\Omega} p_t(x,y) f(y) dy = \int_{d(y, \Omega^c) \leq \varepsilon} p_t(x,y) dy\,.$$
Integrating over $x$ and then exploiting the symmetry of the heat kernel leads to
\begin{align*}
   \int_{\Omega}  (e^{t\Delta}f)(x) dx &= \int_{\Omega} \int_{d(y, \Omega^c) \leq \varepsilon} p_t(x,y) dy dx \\
   &=  \int_{d(y, \Omega^c) \leq \varepsilon} \int_{\Omega} p_t(y,x)  dx dy\,.
\end{align*} 
The inner integral has a probabilistic interpretation: $ \int_{\Omega} p_t(y,x)  dx$ is the likelihood that a Brownian motion started in $y$ and running for $t$ units of time does not exit $\Omega$. At this point, we can invoke basic probability theory. Note first that, by assumption, $\partial \Omega$ is smooth and $\varepsilon>0$ is sufficiently small. This means that the Brownian particle is started very close to the boundary already, at most distance $\varepsilon$, but it also means that, from the point of view of the Brownian particle, the boundary is actually pretty simple: it locally looks like a hyperplane. Instead of following the particle, we can instead simply track its distance to the hyperplane: the projection of a $d-$dimensional Brownian motion onto a one-dimensional line is again a Brownian motion. This has turned the problem into a one-dimensional problem. Luckily, this problem is classical, and is recalled in the next lemma. 
\end{proof}
\begin{lemma}[Reflection principle; e.g.~\cite{katz}] Let
  $\varepsilon > 0$ and let $T_{\varepsilon}$ be the hitting time for one-dimensional Brownian motion started at $B(0) = 0$,
  $T_{\varepsilon} = \inf\left\{ t > 0: B(t) = \varepsilon\right\}.$
  Then
$$ \mathbb{P}(T_{\varepsilon} \leq t) = 2 - 2\Phi\left( \frac{\varepsilon}{\sqrt{t}} \right),$$
where $\Phi$ is the c.d.f. of $\mathcal{N}(0,1)$, i.e.
$$ \Phi(x) = \frac{1}{\sqrt{2\pi}}\int_{-\infty}^x e^{-z^2/2} dz.$$
\end{lemma}
Referring the reader to~\cite{katz} for a proof, let us remark that when $\sqrt{t} \ll \varepsilon$, then $\Phi\left( \varepsilon/\sqrt{t} \right) \sim 1$ and the likelihood of hitting the boundary within $t$ units of time is close to 0. When $\sqrt{t} \gg \varepsilon$, then $\Phi\left( \varepsilon/\sqrt{t} \right) \sim \Phi(0) = 1/2$ and the likelihood of having hit the boundary is close to 1. The transition phase happens at exactly $t \sim \varepsilon^2$ and this can be quantified to any desired level of precision.\\

Another argument would be based on using a result of van den Berg and Davies \cite{berg-davies}. They show that if $\Omega$ has a $C^1-$boundary, if $  1_{\Omega}$ is the indicator function of $\Omega \subset \mathbb{R}^d$ and if $e^{t\Delta} 1_{\Omega}$ is the solution of the heat equation in $\mathbb{R}^d$, then
$$ \int_{\Omega} e^{t\Delta}1_{\Omega}(x) dx = |\Omega| - \frac{2 |\partial \Omega|}{\sqrt{\pi}} \sqrt{t} + \mathcal{O}_{\Omega}(t).$$
The error term can be further refined if $\partial \Omega$ has more regularity, see van den Berg and Le Gall \cite{berg} (but this is not required here). This asymptotic expansion shows that we expect a loss of $\sim \varepsilon$ units of mass within $\varepsilon^2$ units of time. The function $1_{\Omega}(x)$ in van den Berg-Davies and $f(x)$, in our setting, are different. However, using linearity of the heat equation, we have
$$  \int_{\Omega} e^{t\Delta}1_{\Omega}(x) dx  =  \int_{\Omega} e^{t\Delta}f(x) dx +  \int_{\Omega} e^{t\Delta}(1_{\Omega}(x) - f(x)) dx.$$
The function $1_{\Omega}(x) - f(x)$ is supported at distance $\varepsilon$ from the boundary. Coupling this with rapid decay of the heat kernel past scale $\sim \sqrt{t}$ can be used to give an alternative proof.

\subsection{Four short Lemmata}
At this point, we introduce the Dirichlet heat kernel $p_t(x,y)$. We will not use any particular properties of the Dirichlet heat kernel beyond the fact that
$$ (e^{t\Delta}f)(x) = \int_{\Omega} p_t(x,y) f(y) dy\,,$$
and its spectral resolution into Laplacian eigenfunctions (though, one could argue, we used some of its properties implicitly in the proof of Lemma 1).
\begin{lemma}
Let~$f$ denote the initial data in~\eqref{e.specific}. We have
$$  \int_{\Omega} (e^{t \Delta} f)(x) dx = \sum_{k \in A}^{} e^{-\lambda_k t} \left( \int_{d(x, \Omega^c) \leq \varepsilon} \phi_k(x) dx\right)  \left( \int_{\Omega} \phi_k(y) dy\right) \,,$$
\end{lemma}
\begin{proof}
This follows quickly from 
\begin{align*}
  \int_{\Omega} (e^{t \Delta} f)(x) dx =   \int_{\Omega} \int_{d(x, \Omega^c) \leq \varepsilon} p_t(x,y) dy dx\,,
\end{align*}
together with the spectral resolution of the heat kernel
$$  p_t(x,y) = \sum_{k=1}^{\infty} e^{-\lambda_k t} \phi_k(x) \phi_k(y)\,.$$
\end{proof}

The next ingredient is an upper bound on the mean value of an eigenfunction that follows from a result of Hassell-Tao.
\begin{lemma}[Hassell-Tao \cite{hassell}]
We have, for some $c_{\Omega}$ depending only on the domain,
$$ \left| \int_{\Omega} \phi_k(x) dx \right| \leq \frac{c_{\Omega}}{\sqrt{\lambda_k}}\,.$$
\end{lemma}
\begin{proof}
We first argue with the divergence theorem that
$$  \int_{\Omega} \phi_k(x) dx = \frac{1}{\lambda_k} \int_{\partial \Omega} \frac{\partial \phi_k}{\partial n} d\mathcal{H}^{d-1} \leq \frac{\sqrt{|\partial \Omega|}}{\lambda_k}\cdot \left\| \frac{\partial \phi_k}{\partial n}\right\|_{L^2(\partial \Omega)}\,.$$
At this point we invoke Hassell-Tao \cite{hassell}
saying that
$$ \left\| \frac{\partial \phi_k}{\partial n}\right\|_{L^2(\partial \Omega)} \leq c_{\Omega} \sqrt{\lambda_k}\,,$$
from which the claim follows.
\end{proof}

The third Lemma also follows by combining the coarea formula with the Hassell-Tao bound \cite[Lemma 3.2]{hassell}: that lemma yields an upper bound for the~$L^2$ norm of an eigenfunction restricted to the level set~$\{x \in \Omega: d(x,\partial \Omega) = \varepsilon\}.$

\begin{lemma}[Consequence of Hassell-Tao] There exists $0 < \delta \leq C < \infty$, depending only on $\Omega$, such that for all $0 < \varepsilon < \delta$
$$ \int_{d(x, \partial \Omega) \leq \varepsilon} \phi_k(x)^2 \leq C \lambda_k \varepsilon^3$$
\end{lemma}

 \begin{proof}
     By the coarea formula, since the gradient of the distance function has unit magnitude, 
     \begin{multline*}
         \int_{d(x,\partial \Omega) \leqslant \varepsilon} \phi_k(x)^2\,dx =  \int_{d(x,\partial \Omega) \leqslant \varepsilon} \phi_k(x)^2|\nabla d(x,\partial \Omega)|\,dx \\= \int_0^\varepsilon \biggl[ \int_{d(x,\partial \Omega) = r} \phi_k^2(x)\,d\mathcal{H}^{d-1}(r) \biggr]\,dr \leqslant \int_0^{\varepsilon} C\lambda_k r^2\,dr \leqslant C\lambda_k \varepsilon^3\,.
     \end{multline*}
 \end{proof}

The fourth Lemma is completely independent of Laplacian eigenfunctions and deals with estimating a particular infinite series that will show up later in the argument.

\begin{lemma}
Let $c_{\Omega} > 0$ be fixed and $\eta > 0$ be arbitrary. Then there exists $c>0$ (depending only on $d, c_{\Omega}$ and $\eta$) such that, with the choice
$$ B = \frac{c}{\varepsilon^d} \left[ \log\left(\frac{1}{\varepsilon}\right) \right]^{d/2}$$
and for all $\varepsilon$ sufficiently small,
$$  \sum_{k=1 \atop k \geq B}^{\infty}  e^{-c_{\Omega} k^{2/d}  \varepsilon^2} \leq 1\,.$$
\end{lemma}
\begin{proof} We have, for $B \gg 1$,
\begin{align*}
\sum_{ k \geq B}^{}   e^{-c k^{2/d} \varepsilon^2} &\lesssim \int_{B}^{\infty} e^{-c \varepsilon^2 x^{2/d}} dx \leq  \frac{1}{B^{2/d}} \int_{B}^{\infty} e^{-c \varepsilon^2 x^{2/d}}  x^{2/d} dx \\
&= \frac{d}{2} \frac{1}{B^{2/d}} \cdot \varepsilon^{-2-d} \cdot \Gamma\left(1 + \frac{d}{2}, \varepsilon^2 B^{2/d}\right).
\end{align*}
 As long as $ B^{2/d} \varepsilon^2 \gg 1$, we may apply a classical asymptotic (see Abramowitz \& Stegun \cite{AS}) for the incomplete gamma function
$ \Gamma \left(s, x \right) \lesssim_s x^{s-1} \cdot e^{-x}.$
Thus leads us to 
\begin{align*}
\sum_{ k \geq B}^{}   e^{-\lambda_k c_1 \varepsilon^2}  &\leq \frac{d}{2} \frac{1}{B^{2/d}} \cdot \varepsilon^{-2-d} \left( \varepsilon^2 B^{2/d} \right)^{d/2} e^{-\varepsilon^2 B^{2/d}} \\
&\leq c_{\Omega} \cdot \varepsilon^{-2} \cdot B^{1- \frac{2}{d}} e^{-\varepsilon^2 B^{2/d}}.
\end{align*}
 Then, with that choice of $B$, we have $-\varepsilon^2 B^{2/d} \sim - c^{2/d} \log(1/\varepsilon)$ while the term in front is polynomial in $\varepsilon$. By making $c$ sufficiently large, the result follows.
 \end{proof}

\begin{proof}[Proof of the Theorem]
     First, let us assume that $\varepsilon > 0$ is arbitrary (but sufficiently small for all estimates to be applicable). Combining Lemma 1 and Lemma 2 implies that, with absolute constants $ 0 < c_1 < c_2$ and $c_3 >0$ uniformly in $\varepsilon \in (0,\varepsilon_0)$
    $$ \sum_{k \in A}^{} (e^{-\lambda_k c_1 \varepsilon^2} -e^{-\lambda_k c_2 \varepsilon^2}) \left( \int_{d(x, \Omega^c) \leq \varepsilon} \phi_k(x) dx\right)  \left( \int_{\Omega} \phi_k(y) dy\right) \geq c_3 \varepsilon.$$
    At this point, using the triangle inequality, Lemma 3, Lemma 4, Lemma 5 and the Cauchy-Schwarz inequality, up to universal constants depending only on $\Omega$,
    \begin{align*}
        c_3 \varepsilon &\leq \sum_{k \in A}^{} (e^{-\lambda_k c_1 \varepsilon^2} -e^{-\lambda_k c_2 \varepsilon^2}) \left| \int_{d(x, \Omega^c) \leq \varepsilon} \phi_k(x) dx\right|  \left| \int_{\Omega} \phi_k(y) dy\right| \\
        &\lesssim_{} \sum_{k \in A}^{} (e^{-\lambda_k c_1 \varepsilon^2} -e^{-\lambda_k c_2 \varepsilon^2}) \sqrt{\varepsilon} \left( \int_{d(x, \Omega^c) \leq \varepsilon} \phi_k(x)^2 dx\right)^{1/2} \frac{1}{\sqrt{\lambda_k}}\\
         &\leq \sqrt{\varepsilon}\sum_{k \in A}^{}   (e^{-\lambda_k c_1 \varepsilon^2} -e^{-\lambda_k c_2 \varepsilon^2}) \left( \lambda_k \varepsilon^3 \right)^{1/2}  \frac{1}{\sqrt{\lambda_k}}  \,,
    \end{align*}
    and thus for some $c>0$ depending only on $\Omega$ and all sufficiently small $\varepsilon$ 
 $$\sum_{k \in A}^{}     (e^{-\lambda_k c_1 \varepsilon^2} -e^{-\lambda_k c_2 \varepsilon^2})  \geq  \frac{c_3}{\varepsilon}\,.$$   
 These terms are exponentially decaying and it's clear that past a certain index, they will contribute little to the sum. Trivially, for any~$B>0,$
 $$\sum_{k \in A \atop k \geq B}^{}    (e^{-\lambda_k c_1 \varepsilon^2} -e^{-\lambda_k c_2 \varepsilon^2}) \leq \sum_{ k \geq B}^{}   e^{-\lambda_k c_1 \varepsilon^2}\,.$$
At this point, we can invoke Weyl's law $\lambda_k = (1+o(1))\cdot c_{\Omega} \cdot k^{2/d}$ to deduce, for $\varepsilon>0$ sufficiently small, with the choice of~$B$ from Lemma 5 that 
$$ \sum_{\lambda_k \geq c_2 \log\left( \frac{1}{\varepsilon} \right) \frac{1}{\varepsilon^2}}^{}   \frac{1}{\sqrt{\lambda_k}} e^{-\lambda_k c_1 \varepsilon^2} \leq 1\,,$$
 and therefore,
 $$ \sum_{k \in A \atop \lambda_k \leq c_2 \log\left( \frac{1}{\varepsilon} \right) \frac{1}{\varepsilon^2}}   \left(e^{-\lambda_k c_1 \varepsilon^2} -  e^{-\lambda_k c_2 \varepsilon^2}\right)  \geq \frac{1}{2} \frac{c_3}{\varepsilon}\,.$$
 At this point, we quickly derive an elementary inequality. Note that, for $c_1 < c_2$ and all $\varepsilon > 0$,
 $$ e^{- c_1 x} - e^{-c_2 x} = e^{-c_1 x} \left( 1 - e^{ - (c_2 - c_1) x} \right) \leq (c_2 - c_1)x e^{-c_1 x},$$
 where the second inequality follows from the fact that $ 1 - e^{ - (c_2 - c_1) x}$ vanishes in $x=0$ and has derivative $\leq c_2 - c_1$ when $x \geq 0$. Thus
\begin{align*}
\frac{1}{2} \frac{c_3}{\varepsilon}  &\leq  \sum_{k \in A \atop \lambda_k \leq c_2 \log\left( \frac{1}{\varepsilon} \right) \frac{1}{\varepsilon^2}}  \left(e^{- c_1 \lambda_k \varepsilon^2} -  e^{- c_2 \lambda_k  \varepsilon^2}\right) \lesssim_{\Omega} \sum_{k \in A \atop \lambda_k \leq c_2 \log\left( \frac{1}{\varepsilon} \right) \frac{1}{\varepsilon^2}}  \lambda_k \varepsilon^2 e^{- c_1 \lambda_k \varepsilon^2}.
\end{align*}
Using $x e^{-c_1 x} \lesssim 1$ for $x \geq 0$, we deduce
and therefore
$$   \frac{c_3}{\varepsilon} \lesssim  \sum_{k \in A \atop \lambda_k \varepsilon^2 \leq c_2 \log\left( \frac{1}{\varepsilon} \right)} 1 =  \# \left\{k \in A: \lambda_k \leq c \log\left( \frac{1}{\varepsilon} \right)\varepsilon^{-2} \right\}\,.$$
Applying the Weyl asymptotic, we have that
$$ \# \left\{k \in \mathbb{N}: \lambda_k \leq c_2 \log\left( \frac{1}{\varepsilon} \right)\varepsilon^{-2} \right\}  \sim c_3  \log\left( \frac{1}{\varepsilon} \right)^{d/2} \frac{1}{\varepsilon^d}=n\,.$$
Setting $\varepsilon = \sqrt{\log{n}} \cdot n^{-1/d}$ then shows the desired result.
\end{proof}

\end{document}